\documentclass[11pt]{amsart}

 \usepackage{amsopn}
 \usepackage{amsmath,amsthm,amssymb}
 \usepackage[hypertex]{hyperref}

 \newcommand{\nc}{\newcommand}
 
 \renewcommand{\aa}{\mathfrak{a} } 
\nc{\bb}{\mathfrak{b} }
 \nc{\cc}{\mathfrak{c} }  \nc{\dd}{\mathfrak{d} } 
    \nc{\ggo}{\mathfrak{g} }
 \nc{\hh}{\mathfrak{h} }  \nc{\ii}{\mathfrak{i} }
 \nc{\jj}{\mathfrak{j} }  \nc{\kk}{\mathfrak{k} }
\nc{\mm}{\mathfrak{m} }   \nc{\nn}{\mathfrak{n} }
\nc{\pp}{\mathfrak{p} }   
\nc{\rr}{\mathfrak{r} } \nc{\sg}{\mathfrak{s} }
 \nc{\sso}{\mathfrak{so} }  \nc{\spg}{\mathfrak{sp} }
 \nc{\ssu}{\mathfrak{su} }  \nc{\ssl}{\mathfrak{sl} }
 \nc{\tog}{\mathfrak{t} }  \nc{\uu}{\mathfrak{u} }
 \nc{\vv}{\mathfrak{v} } \nc{\ww}{\mathfrak{w} }
 \nc{\zz}{\mathfrak{z} }

\nc{\CC}{{\mathbb C}}
 \nc{\DD}{{\mathbb D}}
\nc{\FF}{{\mathbb F}}
\nc{\GG}{{\mathbb G}}  
\nc{\HH}{{\mathbb H}}
\nc{\II}{{\mathbb I}}
\nc{\JJ}{{\mathbb J}}
\nc{\KK}{{\mathbb K}}
\nc{\NN}{{\mathbb N}}

\nc{\RR}{{\mathbb R}}  
 \nc{\ZZ}{{\mathbb Z}}

\nc{\ggob}{\overline{\mathfrak{g}}} 
 
\nc{\glg}{\mathfrak{gl} }
  
\nc{\pca}{\mathcal{P}} \nc{\nca}{\mathcal{N}}
 
 \nc{\vp}{\varphi} \nc{\ddt}{\frac{{\rm d}}{{\rm d}t}}
 \nc{\la}{\langle} \nc{\ra}{\rangle}
 \nc{\brg}{[\,,\,]_{\ggo}}
 \nc{\brv}{[\,,\,]_{\vv}}
 
 \nc{\SO}{{\sf SO}} \nc{\Spe}{{\sf Sp}} \nc{\Sl}{{\sf Sl}}
 \nc{\SU}{{\sf SU}} \nc{\Or}{{\sf O}} \nc{\U}{{\sf U}}
 \nc{\Gl}{{\sf Gl}} \nc{\Se}{{\sf S}} \nc{\Cl}{{\sf Cl}}
 \nc{\Spin}{{\sf Spin}} \nc{\Pin}{{\sf Pin}}

 \nc{\ad}{\operatorname{ad}} 
 \nc{\Ad}{\operatorname{Ad}}
 \nc{\coad}{\operatorname{coad}}
 \nc{\Aut}{\operatorname{Aut}}
 
 \nc{\Der}{\operatorname{Der}}
 \nc{\Deri}{\operatorname{Deri}}
 \nc{\Dera}{\operatorname{Dera}}
 \nc{\End}{\operatorname{End}} 
 \nc{\Iso}{\operatorname{I}}
  \nc{\Ric}{\operatorname{Ric}}
 \nc{\spa}{\operatorname{span}}
  \nc{\tr}{\operatorname{tr}}
\nc{\rank}{\operatorname{rank}} 
\nc{\ct}{\operatorname{T}}
 
 \theoremstyle{plain}
 \newtheorem{thm}{Theorem}[section]
 \newtheorem{prop}[thm]{Proposition}
 
 \newtheorem{lem}[thm]{Lemma}
 
 \theoremstyle{definition}
 \newtheorem{defn}[thm]{Definition}
 
 \theoremstyle{remark}
 \newtheorem*{rem}{Remark}
 
 \newtheorem{exa}[thm]{Example}

\begin{document}
\title[Naturally reductive pseudo-Riemannian spaces]
{Naturally reductive pseudo-Riemannian spaces}

\author{Gabriela P. Ovando}
\address{G. P. Ovando: CONICET and ECEN-FCEIA, Universidad Nacional de Rosario 
\\Pellegrini 250, 2000 Rosario, Santa Fe, Argentina}
\email{gabriela@fceia.unr.edu.ar}


\begin{abstract} A family of naturally reductive pseudo-Riemannian spaces  is constructed
out of the representations of  Lie algebras with   ad-invariant metrics. We 
exhibit peculiar examples, study their geometry
and  characterize the corresponding naturally reductive homogeneous
structure.
\end{abstract}

\thanks{{\it (2010) Mathematics Subject Classification}:  53C50 53C30 57S25
22F30 53C35. }

\thanks{ Current address: Abteilung f\"ur Reine Mathematik, Albert-Ludwigs
Universit\"at 
Freiburg, Eckerstr.1, 79104 Freiburg, Germany.}

\maketitle

\section{Introduction} 
Recent advances in mathematics and mathematical physics have renovated the interest in the geometry of pseudo-Riemannian metrics of general signature. For example, such metrics
constitute the basis  for certain sigma models,
 supergravities, braneworld cosmology, etc.(see \cite{Co,dW,Ff,F-M-P1,O-V,S-S,vP} and references therein).

Some of this work leads to the investigation of homogeneous geodesics (see for
instance \cite{Du3,Lt1,Lt2}) and manifolds in which every geodesic is homogeneous, called
{\em g.o. spaces} (see the survey in \cite{Du1}). For a time, it was believed that every Riemannian g.o. space is naturally
 reductive, until Kaplan \cite{Ka} found the first counterexample, extended
 to the  pseudo-Riemannian case  recently in \cite{D-K}.
There are significant differences between
 the definite situation, started by Vinberg \cite{Vi} and
followed by others (\cite{Ko2, K-V}, etc.), and the indefinite situation in
which  advances were done in \cite{T-V1, G-O2}.

Ambrose and Singer \cite{A-S} achieved  an infinitesimal characterization of connected simply connected
and complete homogeneous Riemannian manifolds in terms 
of a (1,2) tensor, called a homogeneous structure. This  was generalized to the pseudo-Riemannian
case by Gadea and Oubi\~na \cite{G-O1}. 
 For 
instance
some homogeneous structures
of type $\mathcal T_1 \oplus \mathcal
T_3$ characterize homogeneous Lorentzian spaces for which
every null-geodesic is canonically homogeneous \cite{Me1}.
Naturally reductive pseudo-Riemannian spaces are defined by the existence  of
 a homogeneous structure of type $\mathcal T_3$ (\cite{T-V2,G-O2}). Particular examples are
the pseudo-Riemannian symmetric spaces (homogeneous structures of type
$\{0\}$), which in the Lorentzian situation
emerge as supersymmetric supergravity backgrounds \cite{Ff}.
 Furthermore,  it is conjectured that a pseudo-Riemannian g.o. 
space with non-compact isotropy group should be naturally reductive
\cite{Du2}.

 \vskip 5pt
 
In \cite{Ko1} Kostant proved that if $M$ denotes a naturally reductive 
Riemannian space, such that
the action of the isometry group $G$ on $M$ is transitive and almost
effective, then  $G$
 can be provided with  a bi-invariant metric.  Our first result in this
 work (Theorem
(\ref{t1}) below) extends Kostant's Theorem to the pseudo-Riemannian case.
  As consequence, naturally reductive metrics can be produced 
 in the following way. Let $\ggo$ denote a Lie algebra  equipped with  an 
 $\ad$-invariant
 metric $Q$ and let $\hh \subset \ggo$ be a nondegenerate Lie subalgebra. Thus
  one has the following  reductive decomposition
$$\ggo=\hh \oplus \mm \qquad \quad \mbox{ with } \qquad
[\hh,\mm]\subseteq \mm \quad \mbox{ and } \quad \mm=\hh^{\perp}.$$
Let $G$ denote a Lie group with Lie algebra $\ggo$ and endowed with the
bi-invariant metric induced by $Q$ and let $H\subset G$ denote a closed Lie
subgroup with Lie algebra $\hh$. Then the coset space $G/H$  becomes a naturally reductive pseudo-Riemannian
space.

Here we focus on the construction of a family of simply connected naturally
reductive pseudo-Riemannian Lie groups $\mathcal G(D)$, which can be obtained
as follows. For a class of Lie
algebras $\ggo$ which can be provided with an $\ad$-invariant metric, one 
has   the following decomposition as a semidirect sum of vector spaces
$$\ggo=\hh
\oplus \mathcal G(\dd),$$
where $\mathcal G(\dd)$ denotes the Lie algebra of $\mathcal G(D)$ and 
$\ggo$
 is an isometry Lie algebra acting on $\mathcal G(\dd)$ with 
stability Lie algebra $\hh$.
 In general $\mm \neq
 \mathcal G(\dd)$, however as vector spaces they are isomorphic via the  map 
 $\lambda : \mathcal G(\dd) \to \mm$ which induces the metric  of $\mm$ to
 $\mathcal G(\dd)$ making both spaces linearly
 isometric. The metric $\la \,,\,\ra$  induced on 
 $\mathcal G(\dd)$ is defined on the Lie group $\mathcal G(D)$ by 
 translations on the left.

 The algebraic structure of $\ggo$ can be specified. The
 Lie algebra $\mathcal G(\dd)$ is an ideal in $\ggo$ 
defined on the underlying vector space
 $$\mathcal G(\dd)=\dd \oplus \hh^* \qquad \mbox{ (direct sum as vector
 spaces)}$$
 where $\hh^*$ denotes the dual space of $\hh$, which together with $\dd$, 
 obeys the following data:

\begin{itemize}
  \item  a Lie algebra $\hh$ with  $\ad$-invariant
  metric $\la\,,\,\ra_{\hh}$
  \item  a Lie algebra $\dd$ with  $\ad$-invariant
  metric $\la\,,\,\ra_{\dd}$,
  \item  a Lie algebra   homomorphism $\pi: \hh \to \Dera(\dd, \la\,,\,\ra_{\dd})$ from $\hh$ to the Lie
 algebra of skew-symmetric derivations of  $(\dd, \la\,,\,\ra_{\dd})$.
 \end{itemize}

Another proof for the naturally reductivity property of the Lie groups $(\mathcal
G(D), \la \,,\,\ra)$
is done in terms of
homogeneous structures. For example,
$$T_x y= \frac12 \lambda^{-1} [\lambda x, \lambda y]_{\mm} \qquad \quad \mbox{
for all } x,y \in \mathcal G(\dd)$$
 gives rise to a naturally reductive structure on $\mathcal G(D)$, with 
 $\lambda: (\mathcal G(\dd), \la\,,\,\ra) \to (\mm, Q)$ as already
 mentioned.
This formula generalizes that one of
Riemannian nilmanifolds, where such a naturally reductive homogeneous
structure is unique if the nilmanifold has no Euclidean factor \cite{La2}.

We study the geometry and in particular the group of
orthogonal automorphisms of $\mathcal G(D)$, which allows the production 
of   
 examples of naturally reductive
pseudo-Riemannian spaces with compact or noncompact isotropy group.

Finally we analyze  some special cases.  Starting
with $\dd$ which is $k$-step nilpotent  (resp. solvable) then $\mathcal G(\dd)$ turns
out to be  at most
$k+1$-step nilpotent (resp. solvable). This gives rise to the first examples (known to us) of  naturally
reductive metrics on non semisimple Lie groups, which are neither symmetric nor 2-step
nilpotent.

 \section{Isometries on naturally reductive pseudo-Riemannian 
 manifolds}\label{sec2}

 Let $M$ denote a connected manifold with homogeneous pseudo-Riemannian metric
$\la\,,\,\ra$. Let $G$ be any transitive group of isometries with Lie algebra
$\ggo$ and let $H$  be the isotropy subgroup at some point
$o\in M$. Then  $M\simeq G/H$ is called a homogeneous space which is
{\em reductive} if there exists a subspace $\mm\subseteq \ggo$ such that 
\begin{equation}\label{red}
\ggo=\hh \oplus \mm \qquad \quad \Ad(H)\mm \subseteq \mm
\end{equation}
where $\hh$ denotes the Lie algebra of $H$. If $\la\,,\,\ra$ is positive
definite then $M$ is always reductive, that is, a subspace $\mm$
satisfying (\ref{red}) always exists. 
 
 The action is said {\em (almost) effective} if the set of elements in
$G$ acting as the identity transformation is (discrete) trivial. Almost
effectiveness is equivalent to the requirement that the Lie algebra  $\hh$ contains no non
trivial ideal of $\ggo$. Thus the linear representation of 
  the isotropy group $H$ in $T_oM$, the tangent space to the manifold at the point
  $o=eH\in M$, is   faithful. One may identify $\mm$ with $T_oM$ by the map 
$x \to x^{\bullet}_o$,
with 
\begin{equation}\label{idev}
x_o^{\bullet}=\frac{d}{dt}|_{t=0} \exp tx \cdot o.
\end{equation}
 The map $x\to x^{\bullet}$ is induced by an antihomomorphism of Lie algebras, so
$[x^{\bullet},y^{\bullet}]_o=-[x,y]_o^{\bullet}$.  
One may pull back $\la\,,\,\ra_o$ on $T_oM$ to a (non necessarily definite)
metric on $\mm$. Let $x_{\hh}$ and $x_{\mm}$ denote the $\hh$,
resp. $\mm$ component of $x\in \ggo$.
  
   The canonical connection $\nabla^c$ on the reductive space $G/H$  is the unique 
   $G$-invariant affine connection on $M$ 
  such that for any vector $x\in \mm$ and any frame $u$ at the point $o$,
   the curve $(\exp tx)u$ in the   bundle of frames over $M$ 
   is horizontal. The canonical   connection is complete and the set of 
   its geodesics through $o$ coincides 
  with the set of curves of the type $\exp(tx) \cdot o$, with $x\in \mm$.
Explicitly 
\begin{equation}\label{cc}
(\nabla^c_{x^{\bullet}}y^{\bullet})_o:=[x^{\bullet},y^{\bullet}]_o=-[x,y]^{\bullet}_o \quad \longleftrightarrow \quad
-[x,y]_\mm \quad \mbox{ for } x, y\in \mm.
\end{equation}
 This is an affine connection whose torsion is 
 $$T^c(x^{\bullet},y^{\bullet})=-[x^{\bullet},y^{\bullet}],$$
thus via the identifications, $T^c(x,y)=-[x,y]_{\mm}$ for $x,y \in \mm$ and
as  in \cite{K-N} Ch. X, 
the curvature is given by
$$R^c(x,y)z=-[[x,y]_{\hh},z] \qquad \quad x, y, z\in \mm.$$
The tensor fields $T$ and $R$ are parallel with respect to the canonical 
connection. In 
the reductive space $M=G/H$ with reductive  decomposition $\ggo=\hh \oplus
\mm$ as in (\ref{red}) there is a unique
$G$-invariant affine connection with zero torsion having the same 
geodesics 
 as the canonical connection. This connection is called the {\em natural 
 torsion-free connection} on $M$ (with respect  to the decomposition $\ggo=\hh
 \oplus \mm$). 
 
 \begin{defn} \label{def1} A homogeneous pseudo-Riemannian  manifold 
  \noindent $(M=G/H$, $ \la\,,\,\ra)$ is said to be {\em naturally reductive}
 if it is reductive, i.e. there is a  reductive decomposition 
 $$\ggo=\hh \oplus \mm\qquad \mbox{ with }\qquad \Ad(H)\mm \subseteq \mm$$
and 
$$\la [x,y]_{\mm}, z\ra + \la y, [x,z]_{\mm} \ra=0 \qquad \mbox{ for all } x, y,
 z\in \mm.$$
\end{defn}

Frequently we will say that a metric on a homogeneous space $M$ is naturally 
reductive even though it is not naturally reductive with respect to a particular 
transitive group of isometries (see Lemma 2.3 in \cite{Go}).

 Indeed pseudo-Riemannian  symmetric spaces are naturally reductive. Other 
 examples  arise from Lie groups equipped with a bi-invariant metric. 
  
  If $M$ is naturally reductive  the natural torsion-free 
 connection coincides  with the corresponding Levi Civita connection on $M$, which is given by
 $$(\nabla_{x^{\bullet}}y^{\bullet})_o=\frac12 [x^{\bullet}, y^{\bullet}]_o 
 \quad \longleftrightarrow \quad -\frac12 [x,y]_{\mm} \qquad \mbox{ for } x,y \in \mm.$$
  The  geodesics passing through $o\in M$ are of 
 the form  
 $$\gamma(t)=\exp(t x) \cdot o\qquad \quad \mbox{ for some }x\in \mm.$$

The canonical connection on a naturally reductive pseudo-Riemannian space $M=G/H$ is metric and the curvature tensor satisfies the
following identity
\begin{equation}\label{bianchi}
\la R(x,y)z, w\ra= \la R(z,w)x, y\ra \qquad \quad x, y, z, w\in \mm
\end{equation}
which follows after applications of the Jacobi identity and the naturally
reductive condition (see \cite{DA-Z}).

In view of the definition above if one wants to determine if a given metric is or is not naturally reductive,
 one would first have to find all Lie groups $G$ acting transitive on $M$ and in
 presence of a naturally reductive metric, study all $H$-invariant complements
 $\mm$ of $\hh$. The latter task can be simplified by the next theorem,
 proved in its Riemannian version originally by Kostant \cite{Ko1} (see
 also \cite{DA-Z}). A modification of the proof adapted to the indefinite
  case validates the result for the pseudo-Riemannian situation.
 
 \begin{thm} \label{t1} Let $(M=G/H, \la\,,\,\ra)$ denote a  pseudo-Riemannian naturally reductive  space on which
 $G$ acts almost effectively. Let  $\ggo=\hh \oplus \mm$ be a reductive
 decomposition. Then
 $\bar{\ggo}=\mm +[\mm,\mm]$ is an ideal in $\ggo$ whose corresponding analytic
 subgroup $\bar{G}$ is transitive on $M$ and there exists a unique 
 $\Ad$($\bar{G})$-invariant symmetric nondegenerate bilinear form $Q$ on $\bar{\ggo}$ such that
 $$Q(\hh \cap \bar{\ggo}, \mm)=0\qquad \qquad Q|_{\mm\times\mm}=\la\,,\,\ra$$
 where $\hh \cap \bar{\ggo}$ will be the isotropy algebra in $\bar{\ggo}$.
 
 Conversely if $G$ is connected, then for any $\Ad$($G$)-invariant symmetric 
 nondegenerate bilinear form $Q$ on $\ggo$, which is also nondegenerate on
 $\hh$ (and on $\mm=\hh^{\perp}$), the metric on $M$ defined by $\la\,,\,\ra=Q|_{\mm\times
 \mm}$ is naturally reductive. In this case $\ggo=\bar{\ggo}=\mm +[\mm,\mm]$.
 \end{thm}
 
 \begin{proof} Let $x$ denote an element of $\ggo$. If $x\in \hh$ then
 $[x,y]\in \mm$ for all $y\in \mm$ and  clearly $[x,y]\in
 [\mm,\mm]$ if $x\in \mm$. This proves that $\bar{\ggo}$ is an ideal in $\ggo$ and its Lie group is
 transitive on $M$, since $\bar{\ggo}$ contains $\mm$ and any point in $M$
 can be reached by a geodesic $\exp tx \cdot o$ for some $x\in \mm$. Thus
 by replacing $\ggo$ by $\bar{\ggo}$, assume $\ggo=\bar{\ggo}=\mm + [\mm,
 \mm]$. Since one also has $\ggo=\hh \oplus \mm$, the set
 $$S=\{[y,y']_{\hh}\,:\, y,y'\in \mm\}$$
 generates $\hh$ as a vector space. The relation in (\ref{bianchi})
 establishes  that
 \begin{equation}\label{b}
 \la [y,[z,z']_{\hh}], y'\ra=\la [z,[y,y']_{\hh}], z'\ra \qquad \quad y,y',z,z'\in
 \mm.
 \end{equation}
 Each side of (\ref{b}) depends only on $[y,y']$, $[z,z']$ in $S$ and not
 on the choice of $y,y',z,z'\in \mm$. Thus define $Q$ as
 $$Q_{|_{\mm}}:= Q_{|_{\mm\times\mm}}=\la\,,\,\ra\qquad \qquad Q(\hh,\mm)=0$$
 \begin{equation}\label{q}
 \begin{array}{rcl}
 Q([y,y']_{\hh}, [z,z']_{\hh}) & =& -\la [y,[z,z']_{\hh}],y'\ra \\
 &=& -\la [z,[y,y']_{\hh}],z'\ra \qquad y,y',z,z'\in \mm.
 \end{array}
 \end{equation}
 The first equation in (\ref{q}) shows that $Q$ extends uniquely to a linear
 function on $\hh$ in the second variable for fixed $[y,y']_{\hh}\in S$ and
 the second equation in (\ref{q}) shows that $Q$ extends uniquely to a
 symmetric  bilinear form on $\hh \times \hh$. Now we check that $Q$ is
 $\ad$-invariant
 \begin{equation}\label{adinv}
 Q([x,y],z)=-Q(y,[x,z])\qquad \quad \mbox{for all }x, y, z\in \ggo.
 \end{equation}
 For $x,y,z\in \mm$, (\ref{adinv}) follows from the definition of $Q$ and
 the naturally reductivity condition. For $x,y\in \hh$, $z\in \mm$ every side
 in (\ref{adinv}) vanishes, while for $x\in \hh$ and $y,z\in \mm$,
 (\ref{adinv}) follows from the $\Ad$($H$)-invariant condition of
 $\la\,,\,\ra$.
 
 Finally take $x,y\in \hh$ and $z=[w,w']_{\hh}$ where $w,w'\in\mm$. Thus
 $$\begin{array}{rcl}
 Q([x,y], [w,w']_{\hh}) & = & -\la w, [[x,y], w']\ra \\
 & = & \la w, [[w',x],y]]\ra + \la w, [[y,w'],x]\ra \\
 & = & Q([w,[w',x]],y)+ Q([[w,x],w'],y)\\
 & = & Q([x,[w,w']],y)=-Q([x,[w,w']_{\hh}],y)
 \end{array}
 $$
 which proves  (\ref{adinv}) for $x,y,z\in \hh$. We now show that
 $Q_{|_{\hh\times\hh}}$ is nondegenerate. Let $x\in \hh$ and assume that
 $Q(x,y)=0$ for all $y\in \hh$. In particular 
 $$Q(x,[z,z']_{\hh})=0\qquad \quad \mbox{ for 
 all }z,z'\in \mm,$$
 thus $Q(x, [z,z'])=0$ which after the $\ad$-invariance condition already proved, says that 
 $$Q([x,z],z')=0\qquad \forall \, z'\in \mm \quad \Longrightarrow \quad [x,z]=0 \qquad
 \forall \, z\in \mm$$
 and since the action is effective one gets $x=0$.
 
 To prove the converse, notice that $Q$ is nondegenerate on  
 $\bar{\ggo}$ and since it is an ideal, its orthogonal space 
 $$\bar{\ggo}^{\perp}=\{x\in \ggo \,:\,Q(x,y)=0 \, \mbox{ for all } \, y\in
 \bar{\ggo}\}$$
  is also a nondegenerate ideal of $\ggo$ which is contained in $\hh$ and 
  thus $\bar{\ggo}^{\perp}=0$ since the action of $G$ is almost effective. 
 
  Let $\mm=\hh^{\perp}$ and let $\la\,,\,\ra:= Q_{|_{\mm \times\mm}}$. Then $M:=G/H$
  endowed with the $G$-invariant metric induced by $\la \,,\,\ra$ is a 
  naturally reductive space. Since $Q([h,x], \tilde{h})= 
  -Q(x, [h, \tilde{h}])=0$ for all $h,
  \tilde{h}\in \hh$ and $x\in \mm=\hh^{\perp}$, the splitting  $\ggo=\hh
  \oplus \mm$ is a reductive decomposition. Clearly
  $$\la [x,y]_{\mm}, z\ra + \la y,[x,z]_{\mm}\ra=0$$
  follows from the $\ad$-invariant property of $Q$.  
  \end{proof}

 \begin{rem} The action of the Lie group $\bar{G}$ with Lie
 algebra $\bar{\ggo}$ is also transitive whenever $G/H$ is a Lorentzian
 space admitting a homogeneous structure of type 
 $\mathcal T_1 \oplus \mathcal T_3$  and it plays an important role in the results
 in \cite{Me2}. 
 \end{rem}
 
 {\bf Lie algebras with $\ad$-invariant metrics.} An example of a Lie 
  algebra $\ggo$ with an $\ad$-invariant metric is  a 
 semisimple Lie  algebra  together with its Killing form. The
 corresponding Lie group $G$ is a pseudo-Riemannian Einstein space with
 negative scalar curvature. In the other extreme, the abelian Lie group
 $\mathbb T^r \times \mathbb R^s$ endowed with any pseudo-Riemannian
 invariant metric  is flat.
 
 Among others, Lie algebras with $\ad$-invariant metrics can be constructed out from the following data:
 
 \begin{itemize}
 
 \item  a Lie algebra $(\dd, [\cdot, \cdot]_{\dd})$ with an $\ad$-invariant
 metric $\la\,,\,\ra_{\dd}$, 
 
 \item  a Lie algebra $(\hh, [\cdot, \cdot]_{\hh})$ with $\ad$-invariant 
 symmetric bilinear (possibly degenerate) form $\la\, , \, \ra_{\hh}$,
 \item   a Lie algebra homomorphism $\pi: (\hh, [\cdot,\cdot]_{\hh}) \to \Dera(\dd,
 \la\,,\,\ra_{\dd})$ from $\hh$ to the Lie
 algebra of skew-symmetric derivations of  $\dd$. 
 \end{itemize}
 
 Consider the following vector space direct sum
 \begin{equation}\label{vsd}
 \ggo:= \hh \oplus \dd \oplus \hh^*.
 \end{equation}
 
 Let $Q$ be the symmetric bilinear map on $\ggo$, which for $x_i\in \dd$,
 $\alpha_i\in \hh^*$, $h_i\in \hh$, $i=1,2$, is defined by
   \begin{equation}\label{ad-metric}
 Q((h_1, x_1, \alpha_1 ), (h_2, x_2,\alpha_2)) :=
  \la h_1, h_2\ra_{\hh} +
  \la x_1, x_2\ra_{\dd} + \alpha_1(h_2 ) +\alpha_2(h_1);
  \end{equation}
    it is  nondegenerate and of signature
    $sgn(Q)=sgn(\la\,,\,\ra_{\dd})+(\dim \hh, \dim \hh)$.

   Let $\ad_{\hh}$ denote the adjoint action of $\hh$ to itself. The Lie bracket on $\ggo$ is given by    
\begin{equation}\label{bracketg}
\begin{array}{l}
\qquad \qquad [(h_1, x_1, \alpha_1), (h_2, x_2, \alpha_2)] \\
:=([h_1, h_2]_{\hh},
   [x_1, x_2]_{\dd} + \pi(h_1)x_2 - \pi(h_2)x_1, 
   \beta(x_1, x_2) + \ad_{\hh}^*(h_1)\alpha_2 - \ad_{\hh}^*(h_2)\alpha_1)
   \\ \\
  \quad  \mbox{ where } \quad \beta(x_1,x_2)(h):=\la \pi(h) x_1, x_2\ra_{\dd} \quad \mbox{ and }\quad \ad_{\hh}^*(h)\alpha =-\alpha\circ \ad_{\hh}(h).
   \end{array}
   \end{equation}
   
   Some straightforward computations show that the metric $Q$ in (\ref{ad-metric})
   is $\ad$-invariant with respect to this bracket. 
   
  While  $\hh$ is a subalgebra of $\ggo$,  in general $\dd$ is not a 
  subalgebra.  The subspace $\mathcal G(\dd):=\dd \oplus \hh^*$ is always an
   ideal in $\ggo$, which  is obtained as a central extension of $\dd$ 
   by the 2-cocycle $\beta$ and $\ggo$ is the semidirect sum of $\mathcal G(\dd)$ and $\hh$. 
   
   The resulting Lie algebra $\ggo$  is called the {\em double extension} of $\dd$ with respect to $(\hh, \pi)$.

   One says that a Lie algebra $(\ggo, Q)$ equipped with an
   $\ad$-invariant metric is {\em indecomposable} whenever it has no nondegenerate
   ideal. \footnote{Notice that if $\jj$ is a ideal in $\ggo$ then
   $\jj^{\perp}$ is also a ideal of $\ggo$. Hence if $\jj$ is non 
   degenerate then $\ggo$ admits a direct sum as vector spaces
   decomposition:  
   $\ggo=\jj \oplus \jj^{\perp}$. This shows that the definition of
   indecomposibility given here is equivalent to that one given for
   instance in
   \cite{B-K}.}
                  
      The existence of an $\ad$-invariant metric $Q$ on $\ggo$ impose restrictions in the
      algebraic structure of such a $\ggo$: if 
      $(\ggo, Q)$ is indecomposable, then $\ggo$ is
      either 
      \begin{itemize}
      \item one-dimensional  or 
      \item simple or 
      \item it is a double extension
      of $(\dd, \la\,,\,\ra_{\dd})$ with respect to $(\hh, \pi)$ where
      $\hh$ is either one-dimensional or simple.
      \end{itemize}
      
      See the proof  in \cite{M-R}, and also \cite{F-S} for more details in the
      solvable case.
   
   \begin{exa} If $G$ denotes the corresponding simply connected Lie group $G$
   with Lie algebra $(\ggo,Q)$ and we equip $G$ with the corresponding 
   pseudo-Riemannian metric induced by $Q$ and invariant by left and right
   translations, 
   then $G/H$ is a naturally
   reductive pseudo-Riemannian manifold whenever $H$ is trivial  or $H$ is a
   cocompact discrete subgroup, therefore in the last situation $G/H$ is a
   naturally reductive    pseudo-Riemannian compact space.
   \end{exa}

  \section{The naturally reductive pseudo-Riemannian 
   Lie groups $\mathcal G(D)$}   
   
   In this section we produce
   naturally reductive pseudo-Riemannian spaces. We introduce them as an
   application of Theorem
   (\ref{t1}), but we also find a naturally reductive homogeneous
   structure for them.
   
  Consider the following data
  \begin{itemize} 
  \item  $(\hh, [\cdot, \cdot]_{\hh})$ a Lie algebra with  $\ad$-invariant
  metric $\la\,,\,\ra_{\hh}$, 
  \item $(\dd, [\cdot, \cdot]_{\dd})$ a Lie algebra with  $\ad$-invariant
  metric $\la\,,\,\ra_{\dd}$,
  \item $\pi: \hh \to \Dera(\dd, \la\,,\,\ra_{\dd})$ a Lie algebra 
  homomorphism from $\hh$ to the Lie
 algebra of skew-symmetric derivations of  $\dd$.
 \end{itemize}  
 
 This data enables the construction of a Lie algebra $\ggo$ with an 
 $\ad$-invariant
 metric as explained in the previous section. The vector space
  underlying $\ggo$ is the direct sum
 $\ggo=\hh \oplus \dd \oplus \hh^*$ and which is equipped 
 with the Lie bracket showed in (\ref{bracketg}).  Let $Q_-$ denote the next metric  on  $\ggo$: 
     \begin{equation}\label{-ad-metric}
 Q_-((h_1, x_1, \alpha_1), (h_2, x_2, \alpha_2)) :=
 -\la h_1, h_2\ra_{\hh} +
  \la x_1, x_2\ra_{\dd} + \alpha_1(h_2 ) +\alpha_2(h_1).
  \end{equation}

  Let $G$ be a Lie group with Lie algebra $\ggo$ and let $H$ be a closed
 Lie   subgroup of $G$ with Lie algebra $\hh$. Since $Q_-$ is  an
 $\ad$-invariant metric on $\ggo$ and  $\hh$ is nondegenerate for $Q_-$, 
 Theorem
 \ref{t1} says that  the homogeneous manifold {\bf $G/H$ becomes  a naturally reductive pseudo
 Riemannian space}. In fact one has the following decomposition as direct
  sum as vector 
 spaces 
 $$\ggo=\hh \oplus \mm, \qquad \quad \mm=\hh^{{\perp}},$$
 where $\mm$ is the vector subspace of $\ggo$ $Q_-$-orthogonal to 
 $\hh$:
 $$\begin{array}{rcl}
\mm:=\hh^{{\perp}} & = & \{ x\in \ggo\, : \, Q_-(x, \hh)=0\}\\
& = & \dd \oplus \{ \alpha\in \hh^*, \, h \in \hh \, : \, 
\alpha(\tilde{h}) -   \la h, \tilde{h}\ra_{\hh}=0\quad \forall \, 
\tilde{h}\in \hh\} \\
& = & \dd \oplus \{(h,0,\alpha)\,:\, \alpha (\cdot) =  \la h, \cdot\ra_{\hh}\}
\\ & = & \{(h, x,  h^*)\,/\, h\in \hh, x\in \dd\},
\end{array}
$$ 
where  $h^*$  denotes the image of $h\in \hh$ by the linear isomorphism 
$\ell: \hh \to \hh^*$ 
$$\ell:  h \to 
 h^*:=  \la h, \cdot \ra_{\hh}.
$$
With this notation the coadjoint action of $\hh$ on $\hh^*$ is 
$$\ad_{\hh}^*(h_1)h_2^*=[h_1,h_2]^*_{\hh}$$ 
and  the adjoint action of $\ggo$ restricted to $\hh$ acts on  $\mm$ in the
next way
 $$[h_1, (h_2,x, h_2^*)]=([h_1,h_2]_{\hh},\pi(h_1) x ,  [h_1,h_2]_{\hh}^*)\qquad h_1, h_2\in
 \hh, x\in \dd, $$
which shows that $\mm$ is $\ad(\hh)$-invariant. Since $\hh$ is 
nondegenerate,  the
 restriction of $Q_-$ to $\mm$ (also nondegenerate) is given by
 \begin{equation}
 Q_-((h_1,x, h_1^*), (h_2,x, h_2^*))=\la x,y\ra_{\dd}+  \la h_1, h_2\ra_{\hh}
 \end{equation}
 and it makes  $G/H$  a  naturally reductive space. To define completely the
 metric on $G/H$, as usual identify 
 $T_o(G/H)\simeq \mm$ and define the metric on $g\cdot H$ for $g\in G$ in
 such way that the map $\tau(g): x \cdot H \to gx\cdot H$ becomes an isometry. 
 
 \vskip 5pt
 
 Our goal now is to prove that $(\mm, Q_-)$ is the subspace of 
 Definition (\ref{def1}) 
 corresponding to a Lie group provided  with a  naturally 
 reductive pseudo-Riemannian metric for which left-translations by elements of the group are isometries, sometimes called a left-invariant pseudo-Riemannian metric.
 Such a metric is determined by its values at the identity, hence at the Lie algebra level.
 
 \vskip 3pt
 
  Consider the simply connected Lie group $\mathcal G(D)$ with  Lie algebra
  $$\mathcal G(\dd)= \dd \oplus \hh^*$$ where $\hh$ and $\dd$ are taken as in the data 
   in the beginning of this section and $\hh^*$ is the dual 
  space   of $\hh$.
  The Lie bracket in  $\mathcal G(\dd)$  makes of  the 
 canonical  inclusion $\iota:\mathcal G(\dd) \to \ggo$, 
 $\iota(x+h^*)=(0, x,  h^*)\in \ggo$ a Lie algebra monomorphism:
 \begin{equation}\label{brgd}
 [x_1+ h_1^*, x_2+ h_2^*]=[x_1, x_2]_{\dd} +\beta(x_1,x_2).
 \end{equation}
  
  The vector space underlying the Lie algebra $\mathcal G(\dd)$ is 
   isomorphic to $\mm$ via the linear map 
   $\lambda: \mathcal G(\dd) \to \mm$
   given by
  \begin{equation}\label{lambda}
  \lambda: x +h^* \quad  \to \quad (h,x, h^*).
  \end{equation}
  
 Provide $\mathcal G(D)$  with the pseudo-Riemannian metric 
 $\la\,,\,\ra$  obtained by left-translation of the following metric on 
 $\dd \oplus \hh^*$:  
 \begin{equation}\label{natugd}
 \la x_1+ h_1^*, x_2+h_2^*\ra=  \la h_1, h_2\ra_{\hh} +\la x_1,
 x_2\ra_{\dd}.
 \end{equation} 
 This is the only metric on $\mathcal G(\dd)$ which makes  $\lambda$ a
 linear isometry identifying $\mm$ with $\mathcal G(\dd)$, 
 $$\la x_1+ h_1^*, x_2+h_2^*\ra= Q_-((h_1, x_1,  h_1^*),(h_2,x_2, h_2^*))
 \qquad h_1,h_2\in \hh,\,x_1,x_2\in \dd.$$
 
 With respect to this metric $\hh^*$ and $\dd$ are nondegenerate subspaces of $\mathcal G(\dd)$  which are
 at the same time, orthogonal and complementary to each other. 
 Furthermore since 
 $$Q_-((h,0, h^*), (\beta^*(x_1, x_2),  0, \beta(x_1, x_2))
 =Q_-(h, \beta(x_1,x_2))=Q_-(\pi(h) x_1, x_2)$$  
 via application of $\lambda$  one gets the following relation on 
 $\mathcal G(\dd)$
 \begin{equation} \label{cm}
 \la h^*,[x_1,x_2]\ra = \la h^*, \beta(x_1,x_2)\ra_{\hh^*}= \la \pi(h)x_1,
 x_2\ra_{\dd} \qquad \mbox{ for } h\in \hh, \, x_1, x_2\in \dd
 \end{equation}
 where $\la\,,\,\ra_{\hh^*}$ denotes the restriction of $\la
 \,,\,\ra$ to $\hh^*$.
 
 Since $\hh^*\subseteq (\mathcal G(\dd),\la\,,\,\ra)$ is nondegenerate and 
 $\hh^*\subseteq (\ggo,Q_-)$ is degenerate,  the mapping $\iota$ can 
 {\bf not} be an isometry from  
 $\mathcal G(\dd)\to \iota(\mathcal G(\dd))\subseteq \ggo$. However the group 
 $\mathcal G(D)$ acts  as a group of isometries in both pseudo-Riemannian 
 homogeneous spaces $(M,Q_-)$ and $(\mathcal G(D),
 \la\,,\,\ra)$. 
 
  We shall see that the Lie group $H$ acts by orthogonal automorphisms of 
  $(\mathcal G(D), \la\,,\,\ra)$. Since $\mathcal G(D)$  is simply connected 
  we make no distintion between automorphisms of $\mathcal G(D)$ 
  and $\mathcal G(\dd)$. 
 
 Consider the mapping $\mu$ from $\hh$ to $\End(\mathcal G(\dd))$ induced by the action
 on  $\mathcal G(\dd)= \dd \oplus\hh^*$ taken
 as $\pi$  in the first summand and the coadjoint action
 in the second; that
 is  for all 
 $\tilde{h}, h\in \hh, x\in
 \dd$, the action  is given by
 \begin{equation}\label{mu}
 \tilde{h} \cdot ( x+h^*) :=\pi(\tilde{h}) x + 
 \ad_{\hh}^*(\tilde{h})(h^*) =\pi(\tilde{h}) x + [\tilde{h},h]_{\hh}^*. 
 \end{equation}
 
This is an action which provides a Lie algebra homomorphim from  $\hh$ to
 $\End(\mathcal G(\dd))$ acting  by skew-symmetric derivations with respect
   to $\la\,,\,\ra$. In fact on the one hand
$$h\cdot [x_1+h_1^*, x_2+h_2^*]= \pi(h)[x_1, x_2]_{\dd} +
\ad_{\hh}^*(h)\cdot \beta(x_1, x_2)
 $$
and $ \ad_{\hh}^*(h)\cdot \beta(x_1, x_2)(\tilde{h})= - \beta(x_1, x_2)
[h,\tilde{h}]_{\hh}=  \la \pi([\tilde{h},h]_{\hh})x_1, x_2\ra_{\dd}$ (see (\ref{bracketg})).

On the other hand
$$\begin{array}{rcl}
[ h \cdot (x_1+h_1^*),x_2+h_2^*] & = & 
[\pi(h)x_1 +\ad_{\hh}^*(h)(h_1^*),x_2+h_2^*] \\
& = & [\pi(h)x_1+[h,h_1]_{\hh}^*,x_2+ h_2^*]\\
& = &  [\pi(h)x_1, x_2]_{\dd}+\beta(\pi(h)x_1, x_2)
\end{array}
$$  
and similarly
$$[ x_1+h_1^*,h \cdot (x_2+h_2^*)]= 
[x_1,\pi(h)x_2]_{\dd}+\beta(x_1,\pi(h)x_2).$$
By (\ref{bracketg}), $\beta(\pi(h)x_1, x_2)(\tilde{h})=\la
\pi(\tilde{h})\pi(h)x_1,x_2\ra_{\dd}$ therefore
$$(\beta(\pi(h)x_1, x_2)+
\beta(x_1,\pi(h)x_2))(\tilde{h})=\la \pi([\tilde{h}, h])x_1,x_2\ra.
$$ 
Both the coadjoint representation and $\pi$ are homomorphisms from
$\hh$ to $\End(\hh^*)$ and $\Dera(\dd,\la \,,\,\ra_{\dd})$ respectively,
therefore 
 the map of $\hh$ to $\Der(\mathcal G(\dd))$, denoted $\mu$, is a 
 Lie algebra homomorphism. 
 
The paragraphs above show that we achieved the next result.

\begin{thm} \label{t22} The pseudo-Riemannian metric on the Lie group
 $\mathcal G(D)$
induced by left-translations of $\la\,,\,\ra$ in (\ref{natugd}), makes of it
 a naturally reductive pseudo-Riemannian manifold.

Specifically it admits a transitive action by isometries of the Lie group whose Lie
algebra is the semidirect sum
$$\ggo=\hh \oplus \mathcal G(\dd)\qquad \mbox{ with } \hh \subseteq
\Der(\mathcal G(\dd))\cap \sso(\mathcal G(\dd)):=
\Dera(\mathcal G(\dd)).$$
\end{thm}

\begin{proof} We have already proved that $\hh$ acts by derivations on
$\mathcal G(\dd)$ and it
is not hard to prove that $\mu(h)$ is skew-symmetric with respect to
$\la\,,\,\ra$ for every $h\in \hh$.
Therefore $\hh \subseteq \Dera (\mathcal G(\dd))$ which is  a
subalgebra of the Lie algebra of the isotropy group of isometries of 
$(\mathcal G(\dd), \la\,,\,\ra)$. Since
$\mu:\hh \to \Der(\dd)$ is a homomorphism, actually $\mathcal G(\dd)$ is 
 an ideal in
$\ggo$ so that one gets the semidirect sum structure 
$\ggo=\hh \oplus \mathcal G(\dd)$. Finally the construction of the metric $\la \,,\,\ra$ on $\mathcal G(\dd)$ via $\lambda$ shows that the 
conditions of
(\ref{def1}) hold, and therefore  $(\mathcal G(D),\la\,,\,\ra)$ is a 
naturally
reductive pseudo-Riemannian homogeneous space. 
\end{proof}

\begin{exa} \label{exar2} The low dimensional non-abelian examples of the
 the construction in Theorem (\ref{t22}) arise by starting with the abelian
 Lie algebra  $\aa$ spanned by the vectors $e_1,e_2$. It can be equipped 
 with the  metrics:
 \begin{itemize}
\item  $\la e_1, e_1\ra_+=1=\la e_2, e_2\ra_+$,\quad  $\la e_1, e_2\ra_+=0.$
\item $-\la e_1, e_1\ra_-=1=\la e_2, e_2\ra_-$, \quad  $\la e_1, e_2\ra_-=0.$
\end{itemize}
For the metric $\la \,,\,\ra_+$  the space  of skew-symmetric maps is generated by 
$$t_+=\left(
\begin{matrix}
0 & -1 \\ 1 & 0 \end{matrix} \right)
$$
while for $\la\,,\,\ra_-$ the space  of skew-symmetric maps is generated by
$$t_-=\left(
\begin{matrix}
0 & 1 \\ 1 & 0 \end{matrix} \right).
$$
Let $\RR^{p,q}$ denote the abelian Lie algebra 
equipped with the metric  of signature $(p,q)$. In each 
case considered above, the Lie algebra $\mathcal G(\RR^{p,q})$ has dimension three; one has $\mathcal G(\RR^{p,q})=\aa \oplus \RR
e_3$ and the Lie bracket
$$[e_1, e_2]=e_3$$
 identifies the Heisenberg Lie algebra of dimension three $\hh_3$. The
metric on $\hh_3$ is defined in such way that $e_3$ is orthogonal to $\aa$
giving rise to the next possibilities:

\begin{itemize}
\item $\la e_i, e_i\ra_0 = 1,$ for all i,

\item $\la e_i, e_i\ra_1 = 1=-\la e_3, e_3\ra_1$ \quad  for i=1,2, 

\item $-\la e_1, e_1\ra_2 = 1 = \la e_2, e_2\ra_2 = \la e_3, e_3\ra_2$,

\item $-\la e_1, e_1\ra_3 = 1 = \la e_2, e_2\ra_3 = -\la e_3, e_3\ra_3$.
\end{itemize}

The corresponding pseudo-Riemannian metrics invariant by left-translations,
on the Heisenberg Lie group $H_3$ are naturally reductive. The Lie algebra of
 the isometry group is $\RR \ltimes \hh_3$. 
\end{exa}

\vskip 3pt

\begin{center}{\bf A naturally reductive homogeneous structure on $\mathcal G(\dd)$}
\end{center}

\vskip 7pt

 Our
goal now is to prove the existence of a  homogeneous structure of type 
$\mathcal T_3$ on   the simply
connected Lie group $\mathcal G(D)$. 

Recall that a {\em homogeneous structure} on a connected simply connected
pseudo-Riemannian manifold $(M,g)$ is 
a (1,2)  tensor  field on $M$ satisfying the conditions of the next
theorem, first proved in \cite{A-S}.

\begin{thm}  A connected complete and simply connected Riemannian manifold
$(M,g)$ is homogeneous if and only if there exists a tensor field T of type 
(1,2) on $M$ such that

\vskip 3pt

{\rm (i)} $g(T_x y, z)+g(y,T_xz)=0$

\vskip 2pt

{\rm (ii)} $(\nabla_xR)(y,z)=[T_x,R(y,z)]-R(T_xy,z)-R(y,T_xz)$

\vskip 2pt

{\rm (iii)} $(\nabla_x)T_y=[T_x,T_y]-T_{T_xy}$

\vskip 3pt

for $x,y,z\in \xi(M)$, where $\nabla$ denotes the Levi-Civita connection
of $(M,g)$ and $R$ the corresponding curvature tensor.
\end{thm}

If, in addition, $T$ satisfies the following condition (iv), then $M$ is 
a naturally reductive homogeneous space:

\vskip 3pt

{\rm (iv)} $T_x x = 0$ for all $x\in \xi(M)$.

\vskip 3pt

And the converse holds, that is, if $(M,g)$ is naturally reductive, then
there is a (1,2) tensor on $M$ satisfying (i)-(iv). This was proved in the
Riemannian case by Tricerri and Vanhecke \cite{T-V1} (see also
\cite{T-V2}) and generalized to
the pseudo-Riemannian case by Gadea and Oubi\~na (Proposition 4.1 in 
\cite{G-O2}). 
The tensor $T$ is called a {\em naturally reductive homogeneous structure} on
$M$ or a {\em homogeneous structure of type $\mathcal T_3$}.

Let $\tilde{\nabla}$ define on $(M,g)$ by $\tilde{\nabla}:=T-\nabla$, then the conditions
(i)-(iii) above can be rewriten in the following way:

\vskip 3pt

{\rm (i)'} $\tilde{\nabla}g=0$

\vskip 2pt

{\rm (ii)'} $\tilde{\nabla} R=0$

\vskip 2pt

{\rm (iii)'} $\tilde{\nabla}T=0$.

\vskip 3pt

 Since the metric $\la \,,\,\ra$ on $\mathcal G(D)$
is invariant by translations on the left, to give a naturally reductive homogeneous structure for the
groups $\mathcal G(D)$, it suffices to define such a $T$
 for left-invariant
vector fields, that is, on $\mathcal G(\dd)$. 

In the next section we shall study in more detail some geometric properties of
$\mathcal G(D)$. For now, we make use of the Levi Civita connection
$\nabla$ and the
curvature tensor, which for elements in $\mathcal G(\dd)$ are 
 given by
$$
 \nabla_{x_1+h_1^*}x_2 +h_2^*= \frac12 ([x_1, x_2] -\pi(h_1) x_2
-\pi(h_2) x_1) \qquad  x_i\in \dd, h_i^*\in \hh^*,\, i=1,2.
$$
and $R(w_1,w_2)=[\nabla_{w_1}, \nabla_{w_2}]-\nabla_{[w_1,w_2]}$,
respectively.

\begin{thm} \label{t33} Let $\mathcal G(D)$ be the simply connected Lie 
group endowed
with the pseudo-Riemannian metric obtained by left-translations of  $\la\,,\,\ra$
 on its Lie algebra $\mathcal G(\dd)$ (\ref{natugd}). The following tensor
 $T$ defines a naturally reductive structure on $\mathcal G(D)$:
 \begin{equation}\label{tt1}
 T_x y=\frac12 \lambda^{-1}([\lambda x, \lambda y]_{\mm})
 \end{equation}
 explicitly
 \begin{equation}\label{tt2}
T_{x_1+h_1^*}
x_2+h_2^*=\frac12( [x_1,x_2]+\pi(h_1)x_2-\pi(h_2)x_1 )+[h_1,h_2]_{\hh}^*. 
\end{equation}
\end{thm}
\begin{proof} The  equality (\ref{tt1}) proves (i). In fact 
via the linear isometry $\lambda: \mathcal G(\dd) \to \mm$ condition (i) is equivalent to 
$$Q_-([\lambda x,\lambda y]_{\mm}, \lambda z)+Q_-(\lambda y, [\lambda
x,\lambda z]_{\mm})=0 \qquad \mbox{ for all }x,y,z\in
\mathcal G(\dd).$$

One clearly has $T_x x=0$ for all $x\in \mathcal G(\dd)$.

By taking  $\tilde{\nabla}=T-\nabla$  gives
 \begin{equation}\label{t3}
\tilde{\nabla}_{x_1+h_1^*}
x_2+h_2^*=\pi(h_1)x_2 +[h_1,h_2]_{\hh}^*.
\end{equation}

We show that $\tilde{\nabla}T=0$. On the one hand  one has
\begin{equation}\label{nt1}
\begin{array}{rcl}
\tilde{\nabla}_{x_1+h_1^*}(T_{x_2+h_2^*}x_3+h_3^*)&=&
\tilde{\nabla}_{h_1^*}(\frac12\pi(h_2)x_3-
\frac12\pi(h_3)x_2+\frac12[x_2,x_3]+[h_1,h_2]_{\hh}^*)\\
&=& \frac12 \pi(h_1)\pi(h_2)
x_3-\frac12\pi(h_1)\pi(h_3)x_2+
  \frac12 \pi(h_1)[x_2,x_3]_{\dd}\\ && + \frac12[h_1,
\beta^*(x_2,x_3)]_{\hh}^*+[h_1,[h_2,h_3]_{\hh}]_{\hh}^*
\end{array}
\end{equation}
and on the other hand
\begin{equation}\label{nt2}
\begin{array}{rcl}
T_{\tilde{\nabla}_{x_1+h_1^*}x_2+h_2^*}x_3+h_3^* & = & T_{
\pi(h_1) x_2 +[h_1, h_2]^*_{\hh}}x_3+h_3^*\\
& = &
\frac12[\pi(h_1)x_2,x_3]
-\frac12 \pi(h_3)\pi(h_1) x_2+\frac12\pi([h_1,h_2]_{\hh})x_3\\
& & +[[h_1,h_2]_{\hh},h_3]_{\hh}^*
\end{array}
\end{equation}
\begin{equation}\label{nt3}
\begin{array}{rcl}
T_{x_2+h_2^*}(\tilde{\nabla}_{x_1+h_1^*}x_3+h_3^* )& = & 
T_{x_2+h_2^*}(\pi(h_1)x_3+[h_1,h_3]^*_{\hh})\\
& = &
\frac12[x_2,\pi(h_1)x_3]
-\frac12\pi([h_1,h_3]_{\hh})x_2+\frac12 \pi(h_2)\pi(h_1) x_3\\
& & +[h_2,[h_1,h_3]_{\hh}]_{\hh}^*
\end{array}
\end{equation}
Hence by computing (\ref{nt1})-(\ref{nt2})-(\ref{nt3}) we see that 
$\tilde{\nabla} T$
vanishes: by using  the Jacobi identity in $\hh$, the fact
that $\pi$ is a representation of $\hh$ to $\End(\dd)$ acting by 
skew-symmetric derivations of $\dd$ and the relation
\begin{equation}\label{use}
[h_1, \beta^*(x,y)]_{\hh}^*=\beta(\pi(h_1)x,y)+\beta(x,\pi(h_1)y)
\end{equation}
which derives from the definition of $\beta$. All these facts should be also used
to prove $\tilde{\nabla}R=0$. We exemplify here one case, looking at the
formulas for the curvature tensor in the next section.

\begin{equation}\label{nr1}
\begin{array}{rcl}
\tilde{\nabla}_{x_1+h_1^*}(R(x_2,x_3)h^*)& = & 
-\frac14 [h_1, \beta^*(x_2,\pi(h)x_3)]_{\hh}^*
-\frac14[h_1,\beta^*(\pi(h)x_2,x_3)]_{\hh}^*\\
&& +\frac14 \pi(h_1)\pi(h)[x_2,x_3]_{\dd}
\end{array}
\end{equation}
\begin{equation}\label{nr2}
\begin{array}{rcl}
R(\tilde{\nabla}_{x_1+h_1^*}x_2,x_3)h^*& = & 
-\frac14 \beta(\pi(h_1)x_2,\pi(h)x_3) 
-\frac14 \beta(\pi(h)\pi(h_1)x_2,x_3)\\
&& +\frac14\pi(h)[\pi(h_1)x_2,x_3]_{\dd}
\end{array}
\end{equation}
\begin{equation}\label{nr3}
\begin{array}{rcl}
R(x_2,\tilde{\nabla}_{x_1+h_1^*}x_3)h^*& = & 
-\frac14 \beta(x_2,\pi(h)\pi(h_1)x_3) 
-\frac14 \beta(\pi(h)x_2,\pi(h_1)x_3)\\
&& +\frac14\pi(h)[x_2,\pi(h_1)x_3]_{\dd}
\end{array}
\end{equation}
\begin{equation}\label{nr4}
\begin{array}{rcl}
R(x_2,x_3)\tilde{\nabla}_{x_1+h_1^*}h^*& = & 
-\frac14 \beta(\pi([h_1,h]_{\hh})x_2,x_3) 
-\frac14 \beta(x_2,\pi([h_1,h]_{\hh})x_3)\\
&& +\frac14\pi([h_1,h]_{\hh})[x_2,x_3]_{\dd}
\end{array}
\end{equation}
The expression (\ref{nr1})-(\ref{nr2})-(\ref{nr3})-(\ref{nr4}) vanishes,
which follows from  similar
arguments as above, and  also  from the next relations
$$
\begin{array}{rcl}
[h_1, \beta^*(x_2,\pi(h)x_3)]_{\hh}^* & = &
\beta(\pi(h_1)x_2,\pi(h)x_3)+\beta(x_2,\pi(h_1)\pi(h)x_3)\\ \\
{[h_1, \beta^*(\pi(h)x_2,x_3)]_{\hh}^* }& = & 
\beta(\pi(h_1)\pi(h)x_2,x_3)+\beta(\pi(h)x_2,\pi(h_1)x_3)
\end{array}
$$
\end{proof}

\subsection{Naturally reductive Riemannian nilmanifolds.}
In the next paragraphs we explain the construction of Riemannian nilmanifolds
according to Theorem (\ref{t22}). The translation of our model to that one given in
\cite{La1} is basically due to the equivalence between the  adjoint and the coadjoint
representations on a compact Lie algebra. We also summarize some known results
concerning these spaces. See  \cite{Go, La1} and references therein for more
details.

 The simply connected naturally
reductive Riemannian nilmanifolds without Euclidean factor arise from the
 following data set:

\begin{itemize}
\item a compact Lie algebra $\kk=\bar{\kk}\oplus \cc$ with $\bar{\kk}=[\kk,\kk]$ and
$\cc$ the center of $\kk$;

\item a faithful representation $(\pi,V)$ of $\kk$  without trivial
subrepresentations; 

\item an $\ad_{\kk}$-invariant inner product $\la\,,\,\ra_{\kk}$  
 on $\kk$ and  a $\pi(\kk)$-invariant inner product
$\la\,,\,\ra_{V}$  on $V$.
\end{itemize}

Let $\nn$ denote the vector space direct sum $\nn=V \oplus \kk^*$ 
and  equip $\nn$ with the metric $\la\,,\,\ra$ obtained as in
(\ref{natugd}).
The Lie bracket on 
$\nn= V\oplus \kk^*$ satisfies $[\kk^*,\nn]_{\nn}=0$, $[V, V] \subseteq \kk$
and 

\begin{equation}\label{brac}
\la [u,v], k^*\ra_{\kk^*} = \la \pi(k) u, v\ra_{V}\qquad \quad \mbox{ for } k\in \kk, u, v\in
V.
\end{equation}

Since $(\pi, V)$ has no trivial subrepresentations, the center of $\nn$ 
coincides with $\kk^*$ and since  $(\pi, V)$ is faithful, the commutator 
of $\nn$ is $\kk^*$: $C(\nn)=\kk^*$. 
Take $N$ the simply connected 2-step nilpotent Lie group with Lie algebra 
$\nn$ and endow it with the Riemannian metric invariant by a left-action determined by 
$\la \,,\,\ra$.

 The Lie group $(N, \la\,,\,\ra)$ is 
 naturally reductive and with the notations of this section  it coincides 
 with the group $\mathcal G(V)$ (by identifying $V$ with an abelian Lie
 algebra).

In the case of Riemannian nilmanifolds, the isotropy group
of isometries coincides with the group of orthogonal automorphisms (see
\cite{Wi}). Its Lie algebra $\kk= \Der(\nn)\cap \sso(\nn,
\la\,,\,\ra)$ for naturally reductive simply connected nilmanifolds, 
 is given by 
\begin{equation}\label{nato}
\kk=\bar{\kk} \oplus \uu, \qquad [\bar{\kk},\uu]=0,
\end{equation}
where $\uu=\End_{\kk}(V)\cap \sso(V,\la\,,\,\ra_V)$ and $\End_{\kk}(V)$ denotes the set of intertwinning operators of the
representation $(\pi,V)$ of $\kk$.

\begin{itemize}

\item $\bar{\kk}$ acts on
$\nn= V\oplus \kk^*$ by $(\pi(k), \ad^*(k))$ for all $k\in \bar{\kk}$,  

\item $\uu$ acts trivially on $\kk^*$.

\end{itemize}

 On the simply connected naturally reductive nilmanifold  $(N, \la\,,\,\ra)$
  a naturally reductive homogeneous structure is given by
 
  $$T_{v_1+k_1^*}{v_2+k^*_2}= \frac12(\pi(k_1) v_2-\pi(k_2)v_1) + 
  \frac12 \beta(v_1,v_2) +[k_1, k_2]^*_{\kk}$$
  
  as shown in \cite{La2}. Moreover whenever $N$ has no Euclidean factor,
  such a naturally reductive homogeneous structure is unique.

\section{About the geometry of $(\mathcal G(D), \la\,,\,\ra)$}
 
\vskip 5pt 

 Since $\la\,,\,\ra$ is invariant by left-translations, the covariant derivative $\nabla$ is
 also invariant by left-translations and one can see it as a bilinear map $\nabla:\mathcal
 G(\dd) \times \mathcal G(\dd) \to \mathcal G(\dd)$. To get it explicitly
 one applies the Koszul formula for $w_1, w_2,  w_3\in \mathcal G(\dd)$:
 $$2 \la\nabla_{w_1}w_2, w_3\ra=\la [w_1,w_2],
 w_3\ra-\la[w_2,w_3],w_1\ra+\la[w_3,w_1],w_2\ra.$$

By computing, if $w_i=x_i+h_i^* \in \dd \oplus \hh^*$ for i=1,2,3, one obtains
that the second term in the right side of the Koszul formula is
$$
\begin{array}{rcl}
  \la \beta(x_2, x_3), h_1^*\ra+  \la [x_2,x_3]_{\dd}, x_1\ra 
& = & \la\pi(h_1) x_2, x_3\ra + \la [x_2,x_3]_{\dd}, x_1\ra \\
& = & \la\pi(h_1) x_2, x_3\ra + \la [x_1,x_2]_{\dd}, x_3 \ra 
 \end{array}
$$ 
 and the third term is
 $$
 \begin{array}{rcl}
 \la \beta(x_3, x_1), h_2^*\ra+ \la [x_3,x_1]_{\dd}, x_2\ra 
  & = & \la \pi(h_2) x_3, x_1\ra + \la [x_3,x_1]_{\dd}, x_2\ra\\
  & = & - \la \pi(h_2) x_1, x_3\ra + \la [x_1,x_2]_{\dd}, x_3\ra
 \end{array} 
 $$  
where we are making use of (\ref{cm}) and the fact that $\la \,,\,\ra_{\dd}$ is
$\ad_{\dd}$-invariant. Therefore  the Levi Civita connection is
\begin{equation}\label{levic}
 \nabla_{x_1+h_1^*} x_2+h_2^*= \frac12 ([x_1, x_2] -\pi(h_1) x_2
-\pi(h_2) x_1) \qquad h_i^*\in \hh^*, x_i\in \dd,\, i=1,2.
\end{equation}

Using this, the curvature tensor defined by
$$R(w_1,w_2)w_3 =[\nabla_{w_1},\nabla_{w_2}]w_3-\nabla_{[w_1,w_2]}w_3$$
follows
\begin{itemize}
\item for $x,y,z\in \dd$:
$$R(x,y)z= \frac12 \pi(\beta^*(x,y))z -
\frac14 \pi(\beta^*(y,z))x -\frac14 \pi(\beta^*(z,x))y - \frac14
[[x,y]_{\dd},z]$$
\item for $x_1,x_2\in \dd$, $h^*\in \hh^*$:
$$
\begin{array}{rcl}
R(x_1,x_2) h^* & = & -\frac14 \beta(x_1, \pi(h) x_2) -\frac14 
\beta(\pi(h)x_1, x_2)
+ \frac14 \pi(h) [x_1,x_2]_{\dd}\\ \\
R(x_1,h^*) x_2 & = & -\frac14 [x_1,\pi(h) x_2] +\frac14 \pi(h)[x_1, x_2]_{\dd}\\
& = & [\frac14 \pi(h) x_1,x_2]_{\dd} - \frac14 \beta(x_1, \pi(h)x_2)
\end{array}
$$
\item for $x\in \dd$, $h_1^*, h_2^*\in \hh^*$:
$$
\begin{array}{rcl}
R(x,h_1^*) h_2^* & = & -\frac14  \pi(h_1) \pi(h_2) x\\ \\
R(h_1^*,h_2^*) x & = &  \frac14 \pi([h_1,h_2]_{\hh}) x
\end{array}
$$
\item for $h_i^*\in \hh^*$  i=1,2,3 \qquad $R(h_1^*, h_2^*) h_3^*= 0$, 

 \end{itemize}   
 
 where we use the convention:  if $z\in
 \hh^*$ then $z^*=\ell^{-1}(z)\in \hh$  denotes the image  by
 $\ell^{-1}:\hh^* \to \hh$.
 
 \vskip 3pt
 
 The curvature tensor of $\dd$ equipped with the $\ad_{\dd}$-invariant
 metric $\la\,,\,\ra_{\dd}$ is given by  $R^d(x,y)=-\frac 14
 \ad_{\dd}([x,y]_{\dd})$. Hence the curvature tensor of $\dd$ and $\mathcal G(\dd)$ are related by
 the formula
 $$R(x,y)z= \frac12 \pi(\beta^*(x,y))z - \frac14 \pi(\beta^*(y,z))x -
 \frac14 \pi(\beta^*(z,x))y +\frac14 \beta(z, [x,y]_{\dd}) +R^d(x,y) z.$$
 
 Let $\Pi\subseteq \mathcal G(\dd)$ denote a plane and let $Q$ be the real
 number obtained by computing 
 $$Q(x,y)=\la x, x\ra \la y,y\ra-\la x,y \ra^2\qquad x, y\, \mbox{ basis 
 of } \Pi.$$
 The plane $\Pi$ is nondegenerate if and only if $Q(x,y)\neq 0$ for one
 (hence every) basis of $\Pi$ \cite{ON}. Take $x,y$ be a orthonormal basis
 of a plane $\Pi$, that is, a linearly independent set $\{x,y\}$ such 
 that
 $\la x,y\ra =0$, $\la x,x\ra=\varepsilon_1$, $\la
 y,y\ra=\varepsilon_2$, where $\varepsilon_i=\pm 1$ for i=1,2. The
 sectional curvature of $\Pi$ is given by
 $$K(x,y)= \varepsilon_1 \varepsilon_2\la R(x,y)y, x\ra \qquad
 $$
 which, after the formulas above for the curvature tensor, gets:
 $$
 K(x,y)=\left\{
 \begin{array}{ll}
 \varepsilon_1\varepsilon_2(\frac14 \la [x,y]_{\dd}, [x,y]_{\dd}\ra 
 -\frac34 \la \beta(x,y),\beta(x,y)\ra) & \mbox{ for }x,y\in \dd\\ \\
 \frac{\varepsilon_1\varepsilon_2}4 \la \pi(y^*)x, \pi(y^*)x\ra & \mbox{
 for } x\in \dd, y\in \hh^*\\ \\
 0 & \mbox{ for } x,y \in \hh^*.
 \end{array}
 \right.
 $$  
 
\begin{exa} Let $\la \,,\,\ra_i$ i=1,2, denote the Lorentzian metrics on the
Heisenberg Lie group $H_3$ in Example (\ref{exar2}). 
One verifies that for $\la \,,\, \ra_1$ the sectional curvature 
is nonpositive, while for $\la \,,\,\ra_2$ it could take both positive and negative
values, for instance $K(e_1,e_2) >0$ and $K(e_1,e_3)<0$.
\end{exa} 

 The Ricci tensor, defined as $\Ric(x,y)=\tr( z \to R(z,x)y)$ for
  $x,y \in \mathcal G(\dd)$, becomes a symmetric bilinear form on
 $\mathcal G(\dd)$; this is due to the symmetries of the
 curvature tensor. Hence 
 there exists a symmetric linear transformation $T:\mathcal G(\dd) \to
 \mathcal G(\dd)$ satisfying 
 $$\Ric(x,y)=\la  Tx,y\ra \qquad \quad \mbox{for all }x,y\in
 \mathcal G(\dd)$$ and which is called the {\em Ricci transformation}.
  If $\{e_k\}$ denotes a orthonormal basis of $\mathcal G(\dd)$ such that
  $\la e_k, e_k\ra=\varepsilon_k$, then
  $$\Ric(x,y)=\sum_k \varepsilon_k \la R(e_k,x)y, e_k\ra=\la-\sum_k
  \varepsilon_k \la R(e_k,x)e_k, y\ra$$
         and this implies
	 $$T(x) = -\sum_k  \varepsilon_k  R(e_k,x)e_k \qquad \mbox{ for all }x\in
	 \mathcal G(\dd).$$
Let $\{z_i^*\}$ denote a orthonormal basis of $\hh^*$ with
$\varepsilon_i=\la z_i^*, z^*_i\ra$ and $\{d_j\}$ denote a
orthonormal basis of $\dd$ with $\varepsilon_j=\la d_j, d_j\ra$. 
By computing one obtains

\begin{itemize}
\item for $x\in \dd$, $h^*\in \hh^*$:
$$\Ric(x,h^*)=\frac14 \sum_j \varepsilon_j \la \beta(d_j,[d_j,x]_{\dd}), h^*\ra= 
-\frac14 \tr(\pi(h) \ad_{\dd}(x))$$
\item for $x,y \in \dd$:
$$\Ric(x,y)=\frac12 \sum_i \varepsilon_i\la \pi^2(z_i)x,y\ra-\frac14
\tr(\ad_{\dd}(x)\circ\ad_{\dd}(y))$$
\item for $h_1^*,h_2^*\in \hh^*$:
$$\Ric(h_1^*,h_2^*)= -\frac14 \tr(\pi(h_1) \pi(h_2)).$$
\end{itemize} 
	
To get these relations, one makes use of the formulas for the curvature
tensor. For $x,y\in \dd$ one needs to work a little more. By writing
$\beta(d_j, x)=\sum_i \varepsilon_i \la \beta(d_j,x),z_i^*\ra z_i^*$ one has
$$\begin{array}{rcl}
\sum_j \varepsilon_j \la \pi(\beta^*(d_j,x)) d_j, y\ra & = & 
\sum_j \varepsilon_j \sum_i \varepsilon_i \la \beta(d_j,x),z_i^*\ra
\la \pi(z_i)d_j, y\ra \\
& = & \sum_i \varepsilon_i \la \pi(z_i) ( \sum_j \varepsilon_j  \la \pi(z_i)d_j,x\ra 
d_j),  y\ra\\
& = & \sum_i \varepsilon_i \la \sum_j \varepsilon_j \la \pi(z_i)x,d_j\ra
d_j, \pi(z_i) y\ra \\
& = & \sum_i \varepsilon_i \la \pi(z_i) x, \pi(z_i) y\ra\\
& = & - \sum_i \varepsilon_i \la \pi^2(z_i)x,y\ra
\end{array}
$$

where also $\sum_j \varepsilon_j \la \pi(z_i)x,d_j\ra d_j = \pi(z_i)x.$
Thus the Ricci transformation is

$$
\begin{array}{rcll}
T(h^*) & = & \frac14 \sum_j \varepsilon_j [d_j, \pi(h)d_j] & h^*\in \hh^*\\ \\
T(x) & = & \frac12 \sum_i \varepsilon_i \pi(z_i)^2x - \frac14 \sum_j \varepsilon_j\ad^2(d_j)x
 & x\in \dd.\end{array}
$$

\begin{rem} Note that whenever $\dd$ is abelian, (and therefore $\mathcal
G(\dd)$ is 2-step nilpotent) the Ricci transformation preserves the
decomposition $\dd \oplus \hh^*$, that is $T(\hh^*)\subseteq \hh^*$ and
$T(\dd)\subseteq \dd$.
\end{rem}

Let $\mathcal S(D)$ denote a connected subgroup of $\mathcal G(D)$ with Lie
algebra $\mathcal S(\dd)\subseteq \mathcal G(\dd)$. Since translations on
the left by elements of $\mathcal S(D)$ are
isometries of $\mathcal G(D)$ then  $\mathcal S(D)$ 
is  totally geodesic if and only if it is totally geodesic at the
identity. Therefore
$\mathcal S(D)\subseteq \mathcal G(D)$ is a totally
geodesic submanifold of $\mathcal G(D)$ if and only if $\nabla_u v\in
\mathcal S(\dd)$ whenever $u,v\in \mathcal S(\dd)$.

\vskip 3pt

\begin{defn} A Lie algebra $\mathcal S(\dd)\subseteq \mathcal G(\dd)$ is
totally geodesic if $\nabla_u v\in \mathcal S(\dd)$ for all $u,v\in
\mathcal S(\dd)$.
\end{defn}
 
 \begin{exa} The Lie subalgebra $\hh^*$ is a flat totally real Lie algebra.
 \end{exa}

\begin{exa} Let $\xi=x+h^*\in \mathcal G(\dd)$ be arbitrary. The
1-parameter subgroup $\exp(t\xi)$ is a geodesic if and only if $\pi(h)x=0$.
In particular if $\xi\in \hh^*$ or $\xi\in \dd$ the curve $t\to \exp(t\xi)$
is a geodesic starting at the identity.
\end{exa}

Now we would like to investigate some  groups  acting by isometries on  
$(\mathcal G(D), \la\,,\,\ra)$. Let $\Iso(\mathcal G(D))$ denote the group of all isometries of
$\mathcal G(D)$.   It is well known that the group $\Iso(\mathcal G(D))$, endowed with the compact-open
topology, is a Lie group, even a Lie transformation group of $\mathcal G(D)$. 
For $u \in \mathcal G(D)$ let $L_u$ denote the left-translation by $u\in \mathcal G(D)$ and let $\mathcal L$ denote the
subgroup of all left-translations. Let $\mathcal F$ denote the
stability subgroup of the identity element $e$ (see \cite{Mu}).

 \begin{lem} $\mathcal F$ is a closed subgroup of $\mathcal G(D)$,
 $\mathcal L$ is a closed connected and simply connected Lie subgroup
 isomorphic to $\mathcal G(D)$ and
 $\Iso(\mathcal G(D))=\mathcal F \mathcal L$, where $\mathcal L \cap \mathcal
 F=\{id\}$. 
 Each element of $\mathcal F$ is determined by its differential at
 $e$. 
 
 Let $\Or_{aut}(\mathcal G(D))$ denote the group of isometric automorphisms
 and let $\Iso_{spl}(\mathcal G(D))$ denote the subgroup of 
 $\Iso(\mathcal G(D))$  preserving the splitting $T\mathcal G(D)=\hh^* \mathcal G(D) \oplus \dd
 \mathcal G(D)$.
 
 Then $\Iso_{aut}(\mathcal G(D))= \Or_{aut}(\mathcal G(D)) \oplus \mathcal G(D)$ 
 semidirect sum and  it holds
  $$\Iso_{aut}(\mathcal G(D))\subseteq \Iso_{spl}(\mathcal G(D)) 
  \subseteq \Iso(\mathcal  G(D)).$$
   \end{lem}

We shall study here the case of  isometries preserving the subspaces 
$\hh^*$ and $\dd$. 

Let $\varphi$ denote an isometry which leaves invariant the subspaces 
$\hh^*$ and $\dd$; 
write  $\varphi= \psi + \phi$ where $\psi\in\Or(\hh^*,\la\,,\,\ra_{\hh^*})$  and $\phi\in \Or(\dd,
\la\,,\,\ra_{\dd})$. It 
is clear that the restriction of $\varphi$ to $\dd$ must be an
isometry of $\dd$, hence according to \cite{Mu}, $\phi$ must satisfy 
$$\phi[x,[y,z]]_{\dd}=[\phi x, [\phi y, \phi z]]_{\dd} \qquad \mbox{ for all } x, y, z
\in \dd.$$ 
By using this, whenever $\psi [x,[y,z]]_{\hh^*}=[\phi x, [\phi y, \phi
z]]_{\hh^*}$
for all $x,y,z\in \dd$, by  computing $\la \psi[x,[y,z]]_{\hh^*}, \psi h^*\ra= \la
[x,[y,z]], h^*\ra$ one gets

\vskip 4pt 

$\la \pi(\psi h^*) \phi x, [\phi y, \phi z]_{\dd}\ra = \la \pi(h) x,
[y,z]_{\dd}\ra$ and therefore 

\vskip 4pt 

$\ad_{\dd}(\phi z) \circ  \pi( \psi h)=\phi \circ \pi(h) \circ
\psi^{-1} \quad \mbox{ for all } h\in \hh, z\in \dd.$

\vskip 4pt

Now we compute the group of orthogonal automorphisms  of $\mathcal G(D)$.
 Since $\mathcal G(D)$ is simply connected we see it at the Lie algebra 
 level. 
Let $\varphi:\mathcal G(\dd) \to \mathcal G(\dd)$ be  a
orthogonal automorphism. Since $\varphi$ preserves the center,  as above 
we can
write it as $\varphi=\psi + \phi$. One computes that  
$\la [\phi x_1, \phi x_2], \psi h^*\ra = \la [x_1,x_2], h^*\ra$
so one has 
$$\begin{array}{rcl}
\la \beta(\phi x_1, \phi x_2), \psi h^*\ra & = & \la \beta(x_1,x_2), h^*\ra
\end{array}
$$ getting
 \begin{equation}\label{is2}
 \pi(\psi h) =\phi\pi(h)\phi^{-1} \quad \quad \mbox{ for all } h\in \hh,
 \end{equation}
while
$$\la [\phi x_1, \phi x_2], \phi x_3 \ra_{\dd} = \la [x_1,x_2], x_3 \ra_{\dd}$$
implies that $\phi\in \Or(\dd, \la\,,\,\ra_{\dd})\cap \Aut(\dd)$.

\begin{prop} \label{iso} The group of orthogonal automorphisms of $\mathcal G(\dd)$
consists of elements
$$\psi+\phi\in \Or(\hh^*,\la\,,\,\ra_{\hh^*}) \times 
(\Or(\dd, \la\,,\,\ra_{\dd})\cap \Aut(\dd))\,:\, 
\phi \pi(h) \phi^{-1}=\pi(\psi h)\mbox{ for all } h\in \hh.$$

Its Lie algebra $\sso_{aut}(\mathcal G(\dd))$ is therefore
$$A+B \in \sso(\hh^*,\la\,,\,\ra_{\hh^*}) \times 
\Dera(\dd)\quad :\quad 
[B, \pi(h)]=\pi(Ah) \quad \mbox{ for all } h\in \hh.$$
In particular $\pi(\hh) \subseteq \sso_{aut}(\mathcal G(\dd))$.

\end{prop}

\begin{rem} The isometry group of a bi-invariant metric was studied 
in \cite{Mu}. Isotropy groups of compact and non compact type can be
obtained starting with an abelian Lie algebra provided with different
metrics. The resulting Lie group is 2-step nilpotent. Moreover whenever the center
is nondegenerate the stability group coincides with the group of orthogonal
automorphisms \cite{C-P} (see \cite{Ov} for  examples). 
\end{rem}

\section{Examples and applications}
The examples considered in (\ref{exar2}) appeared in \cite{Ov} while naturally
reductive Riemannian nilmanifolds were extended studied in the past. For those cases 
(\ref{t22}) and (\ref{t33}) offer a new point of view.  
 An original contribution   is the production of several
 Lie groups equipped with naturally reductive metrics (and the corresponding homogeneous structures) which have no analogous
 in the Riemannian case.  Explicitly we shall
 construct examples in dimension six which are 4-step nilpotent. We shall also
 study the algebraic structure of these objects.

Let $\mathfrak a(p,q)$ denote the double extension Lie algebra of $\RR^{p,q}$
via the skew-symmetric map $A\in \sso(p,q)$.  Thus $\mathfrak a(p,q)$ is
solvable and scalar flat (see \cite{B-K}). It is indecomposable if $\ker A$ 
is totally  isotropic.
 
Let $\mathcal G(\RR^{p,q})$ denote the naturally reductive Lie algebra constructed
fron $\RR^{p,q}$ via $A$ after Theorem (\ref{t22}). Let $z$
be an element generating the one-dimensional vector space complementary to
$\RR^{p,q}$. The Lie bracket on $\mathcal G(\RR^{p,q})$ is given by
$$[x,y]=\la Ax, y\ra z \qquad \quad \mbox{ for all }x,y\in \RR^{p,q}.$$
 
 The center of $\mathcal G(\RR^{p,q})$ is
 $$\zz(\mathcal G(\RR^{p,q}))=\RR z \oplus \ker A$$
 and if $\tilde{\zz} \subseteq \ker A$ is a nondegenerate subspace in 
 $\RR^{p,q}$ then $\mathcal G(\RR^{p,q})$ is decomposable.

\begin{prop} Any Lie algebra $\mathcal G(\RR^{p,q})$ is 2-step nilpotent. Moreover 

\begin{itemize}

\item If $A$ is nonsingular then $\mathcal G(\RR^{p,q})$ is the Heisenberg Lie algebra
$\hh_{2s+1}$ with $2s=p+q$.

\item If $A$ is singular then  $\mathcal G(\RR^{p,q})$ is a 
central extension of a Heisenberg Lie algebra $\hh_{2s+1}$ with $2s \le p+q$,
and $\mathcal G(\RR^{p,q})$ is indecomposable  if
$\ker A$ is totally isotropic.
\end{itemize}
The signature of the metric on $\mathcal G(\RR^{p,q})$ is $(p,q+1)$, or 
$(p, q+1)$, depending on the sign  of  $z$.

Moreover if the center of $\mathcal G(\RR^{p,q})$ is nondegenerate the 
stability-group of isometries consists of the orthogonal automorphisms 
\cite{C-P}, hence they can be computed as in (\ref{iso}).
\end{prop}

\begin{exa} The  isometry groups  for  the naturally reductive spaces $(H_3, \la\,,\,\ra_i)$ 
 i=0,1,2 in (\ref{exar2}) are

\vskip 3pt

$\Or(2) \oplus H_3$ for the metrics $\la\,,\,\ra_i$ i= 0,1;

\vskip 3pt

$\Or(1,1) \oplus H_3$ for the metric $\la\,,\,\ra_2$

\vskip 3pt

where $\Or(2)$ denote the orthogonal group of $\RR^2$ and $\Or(p,q)$ the
orthogonal group for $\RR^{p,q}$. On each isometry Lie algebra $\RR \oplus
\hh_3$ equipped with an $\ad$-invariant metric, every skew-symmetric
derivation is an inner derivation. 
\end{exa}

\begin{exa} \label{f3} The free 3-step nilpotent Lie algebra in two generators can be
obtained by a double extension procedure: take $\RR^{1,2}$ and the
skew-symmetric map $A$ with matrix
$$\left( \begin{matrix}
0 & 1 & 0 \\
1 & 0 & 1\\
0 & -1 & 0
\end{matrix}
\right).
$$
In this way we get the Lie algebra $\mathfrak a(1,2)$ spanned by the vectors
$e_0, e_1, e_2, e_3, e_4$ which satisfy the Lie bracket relations
$$[e_4,e_1]=e_2\qquad [e_4,e_2]=e_1-e_3 \qquad [e_4,e_3]=e_2$$
$$[e_1,e_2]=e_0 = [e_3, e_2].$$
An $\ad$-invariant metric can be defined on $\mathfrak a(1,2)$ obeying the rules
$$\la e_1, e_1\ra=-1 \qquad \la e_2,e_2\ra=\la e_3, e_3\ra=\la e_4,
e_4\ra=\la e_0, e_4\ra=1.$$
The kernel of $A$ is the vector space spanned by $e_1-e_3$ which is totally
isotropic. Thus 
$$\mathcal G(\RR^{1,2})=\RR z \oplus \hh_3$$
is   indecomposable. 

Note that the corresponding simply connected Lie group whose
Lie algebra is $\mathcal G(\RR^{1,2})$, admits two nilpotent non-isomorphic Lie groups acting transitively on it.

The resulting metric on $\mathcal G(\RR^{1,2})$ is Lorentzian if the metric on 
$z$ is
considered positive, while it is neutral if
the metric on $z$ is chosen negative.
\end{exa}

To get other kind of Lie algebras $\mathcal G(\dd)$ which are  not
2-step nilpotent, we should need nonabelian Lie algebras $\dd$. 
But moreover one also needs that the set of skew-symmetric derivations of
such a $\dd$  consists not only of inner derivations.
According to investigations in \cite{F-S}, if $A, B$ are skew-symmetric 
derivations of $\dd$, then the double extension
Lie algebras $\ggo_A$ and $\ggo_B$ getting from $A$ and $B$ respectively are
isomorphic and isometric if and only if  there exist $\lambda \in \RR-\{0\}$,
$T\in \dd$ and $\varphi\in \Aut(\dd)$ such that
$$\varphi B \varphi^{-1}=\lambda A + \ad(T).$$

The proof of the next lemma follows by computations following the
 definitions. Notice that we choose a different $\ad$-invariant metric
  from that in (\ref{f3}).

\begin{lem} \label{dero} Let $\aa(1,2)$ denote the free  3-step nilpotent Lie algebra in
two generators as in (\ref{f3}).  
Then the Lie algebra of skew-symmetric derivations is the semidirect sum
 $$\Dera(\aa(1,2))=\mathfrak{sl}(2,\RR)
\oplus \hh_3 $$ 
where the Lie algebra of inner derivations of $\aa(1,2)$ is isomorphic 
to   $\hh_3$, the Heisenberg Lie algebra of dimension three.
\end{lem}

\begin{proof}
In the ordered basis $\{e_0,e_1-e_3, e_2,e_1, e_4\}$
 any skew-symmetric derivation has a matrix of the form
$$\left(
\begin{matrix}
a & b & x & z & 0 \\
c & -a & y & 0 & z \\
0 & 0 & 0 & y & -x\\
0 & 0 & 0 & a & b \\
0 & 0 & 0 & c & -a
\end{matrix}
\right) \qquad \quad a,b,c,x,y,z\in \RR.$$
Denote by $e_{ij}$ the $5\times 5$ matrix whose entries are all 0 except
for the entry $1$ at the position $(ij)$. Consider the
matrices
$$
\begin{array}{rcl}
H & = & e_{11}-e_{22}+e_{44}-e_{55} \\
E & = & e_{12} + e_{45}\\
F & = & e_{21}+ e_{54}\\ \\
X & = & e_{13}-e_{35} \\
Y & = & e_{23}+e_{34} \\
Z & = & e_{14} + e_{25},
\end{array}
$$
then provided with the usual Lie bracket of matrices, one has 
$\spa\{H,E,F\}=\mathfrak{sl}(2,\RR)$, $\spa\{X,Y,Z\}=\hh_3$
and  $[H,X]=X$, $[H, Y]=-Y$, $[E,Y]=X$, $[F,X]= Y$. Therefore the Lie algebra of skew-symmetric derivations is 
 $$\Dera(\aa(1,2))=\mathfrak{sl}(2,\RR)
\oplus \hh_3 \qquad \mbox{ semidirect sum}.$$ 

\vskip 2pt

It is not hard to prove that the Lie algebra of inner derivations
is 

$\Deri(\aa(1,2))=\spa\{X,Y,Z\}$.
\end{proof}

Each of the Lie algebras resulting from the construction of Theorem
(\ref{t22}) by  $H$, $E$, $F$  is  4-step nilpotent. 
We shall study this in more detail.

For every Lie algebra $\ggo$, recall that the derived $\{C^i(\ggo)\}$ and 
descending central $\{D^i(\ggo)\}$ 
series are inductively defined by
$$
C^0(\ggo)  =  \ggo = D^0(\ggo)
$$
$$C^i(\ggo)=[C^{i-1}\ggo, C^{i-1}\ggo] \qquad \quad D^i(\ggo)=[\ggo, D^{i-1}\ggo]\qquad
i\geq 1.
$$
A Lie algebra is called {\em k-step} solvable if there exists $k$ such that 
$C^k(\ggo)=0$ but $C^{k-1}(\ggo) \neq 0$. 

Let $\mathcal G(\dd)$ denote the Lie algebra arising from the construction in  (\ref{t22}). From the definitions above
one has
$$C^0(\mathcal G(\dd))=\mathcal G(\dd)\qquad C^1(\mathcal G(\dd))=\{[x,y]\,:\,
x,y \in \dd\} \subseteq C^1(\dd)\oplus \hh^*,$$
and inductively one verifies  
$$C^i(\mathcal G(\dd))\subseteq C^i(\dd) \oplus \hh^* \qquad \mbox{ for } i \geq
1.$$ 
Therefore if $\dd$ is k-step solvable, $C^k(\dd)=0$ and hence
$C^{k}(\mathcal G(\dd))\subseteq \hh^*$ giving 
$$C^{k+1}(\mathcal G(\dd))=0$$
so that $\mathcal G(\dd)$ is at most $k+1$-step solvable. 
Notice that $\mathcal G(\dd)$ could be $k$-step solvable if for all $x,y\in
C^{k-1}(\dd)$ it holds $\beta(x,y)=0$, that is
\begin{equation}\label{sol}\la \pi(h) x, y \ra_{\dd}=0 \qquad \mbox{ for all } h\in \hh, \, x,y \in
C^{k-1}(\dd).
\end{equation}

A Lie algebra $\ggo$ is said to be {\em k-step} nilpotent if $D^k(\ggo)=0$ but
 $D^{k-1}(\ggo)\neq 0$. By computing one has
 $$D^0(\mathcal G(\dd))=\mathcal G(\dd)\qquad D^1(\mathcal G(\dd))=\{[x,y]\,:\,
x,y \in \mathcal G(\dd)\} \subseteq D^1(\dd)\oplus \hh^*,$$
and in general by induction one gets
$$D^i(\mathcal G(\dd))\subseteq D^i(\dd) \oplus \hh^* \qquad \mbox{ for } i \geq
1.$$
 So if $\dd$ is $k$-step nilpotent, $\mathcal G(\dd)$ is at most $k+1$-step
 nilpotent. Moreover, $\mathcal G(\dd)$ is $k$-step nilpotent if and only if
 $$\la \pi(h) x, y\ra_{\dd}=0 \qquad \mbox{ for all } h\in \hh, x\in D^{k}(\dd), 
 y\in \dd,
 $$
equivalently 
\begin{equation}\label{nil}
D^k(\dd) \subseteq  \bigcap_{h\in \hh} \ker \pi(h).
\end{equation}

\begin{prop} Let $\dd$ denote a  Lie algebra with $\ad$-invariant
metric $\la \,,\,\ra_{\dd}$ and let $\mathcal G(\dd)$ denote the Lie algebra $\mathcal
G(\dd)=\dd\oplus \hh^*$ with the Lie bracket as in (\ref{t22}). Then

\begin{itemize}

\item If $\dd$ is $k$-step solvable, $\mathcal G(\dd)$ is either $k$- or $k+1$-step
solvable. It is $k$-step solvable if and only if  (\ref{sol}) holds.

\item If $\dd$ is $k$-step nilpotent, $\mathcal G(\dd)$ is either $k$- or $k+1$-step
nilpotent. It is $k$-step nilpotent if and only if  (\ref{nil}) holds.
\end{itemize}
\end{prop}

\begin{exa} Let $\{z, e_0, e_1-e_3, e_2, e_1, e_4\}$ denote a basis of 
the central extension of $\aa(1,2)$ as in (\ref{t22}) by the cocycle induced by one of the skew-symmetric derivations $H, E, F$ in Lemma
(\ref{dero}). For each extension one has
$$[e_4, e_2]=e_2 \qquad [e_4, e_2]=e_1-e_3 \qquad [e_1, e_2]=e_0,$$
and the additional Lie brackets follow
$$\begin{array}{llc}
\mathcal G_H(\aa(1,2)) & &-[e_4, e_0]= z =[e_1, e_1-e_3]\\ \\
\mathcal G_E(\aa(1,2)) & & [e_4, e_1-e_3]=z\\ \\
\mathcal G_F(\aa(1,2)) & & [e_1, e_1-e_3]=z
\end{array}
$$
giving rise to 4-step nilpotent Lie algebras. The corresponding Lie groups
provided with  suitable pseudo Riemannian metrics as in Theorem (\ref{t22}) are naturally reductive
but nonsymmetric. 
\end{exa}

\vskip 5pt

{\bf Acknoledgements.}   The author is grateful to the referee for
suggesting modifications to the original version in order to get improved
results;    to V.
Bangert, Mathematisches Institut der Albert-Ludwigs Universit\"at Freiburg,
for his support during the stay of the author in this institution, where
the main part of this work was done. 
 
 For motivating discussions and for highlighting a deeper geometry many thanks to 
 A. Kaplan.

\end{document}